\documentclass[11pt]{amsart}

\usepackage{amssymb,amsthm,amsmath}

\usepackage{mathtools}
\usepackage{hyperref}
\RequirePackage[dvipsnames,usenames]{xcolor}

\input{kmacros3.sty}

\usepackage{tikz}
\usepackage{tikz-cd}
\usetikzlibrary{cd}
\usepackage{graphicx}
\usepackage[all,cmtip]{xy}

\usepackage{mabliautoref}
\numberwithin{equation}{theorem}

\usepackage{bm}
\usepackage{pifont}
\usepackage{upgreek}

\usepackage{eucal}
\usepackage{ulem}
\usepackage{stmaryrd}

\numberwithin{equation}{theorem}

\usepackage{fullpage}

\usepackage{enumerate}
\usepackage{calc}

\usepackage{verbatim}
\usepackage{alltt}
\usepackage{scalerel,stackengine}

\begin{document}
\title{Factoring maps to big Cohen-Macaulay algebras through blowups}
\author{Karl Schwede}
\address{Department of Mathematics, University of Utah, Salt Lake City, UT 84112, USA}
\email{schwede@math.utah.edu}


\begin{abstract}
We show that for a reduced universally catenary equidimensional Noetherian local ring $R$, any sufficiently functorial big Cohen-Macaulay algebra assignment $R \mapsto B_R$, and any proper birational map $Y \to \Spec R$,  there exists a factorization $R \to \myR\Gamma(Y, \cO_Y) \to B_R$ in the derived category of $R$-modules.  
\end{abstract}

\maketitle

\section{Introduction}

Big Cohen-Macaulay modules and algebras were identified by Hochster as a tool for proving deep results in commutative algebra \cite{HochsterCMModules,HochsterTopicsInTheHomologicalTheory}, \cf \cite{Sharp.CMPropertiesForBalancedBigCM}.  Big Cohen-Macaulay algebras associated to an excellent Noetherian local domain $R$ were shown to exist in characteristic $p > 0$ in \cite{HochsterHunekeInfiniteIntegralExtensionsAndBigCM} where Hochster and Huneke showed that $R^+$, the integral closure of $R$ in an algebraic closure of its fraction field, is big Cohen-Macaulay.  From this the existence of big Cohen-Macaulay algebras in characteristic zero can also be deduced.  Also see \cite{HochsterHunekeApplicationsofBigCM,HunekeLyubeznikAbsoluteIntegralClosure,SmithTightClosureParameter}.  In mixed characteristic, big Cohen-Macaulay algebras were shown to exist by Andr\'e in \cite{AndreDirectsummandconjecture,AndreWeaklyFunctorialBigCM}, see also \cite{GabberMSRINotes,HeitmannMaBigCohenMacaulayAlgebraVanishingofTor}.  More recently, Bhatt showed that the $p$-adic completion $\widehat{R^+}$ of $R^{+}$ is big Cohen-Macaulay in mixed characteristic \cite{BhattAbsoluteIntegralClosure}, also see \cite{SchoutensCanonicalBCMAlgebras} in characteristic zero.

An interesting feature of $\widehat{R^+}$ in positive and mixed characteristic is that if $\pi : Y \to \Spec R$ is any proper birational map, one has a factorization 
\[
    R \to \myR\Gamma(Y, \cO_Y) \to \widehat{R^+}.
\]
In characteristic $p > 0$, this was shown in \cite{BhattDerivedDirectSummand} and in mixed characteristic  it follows quickly from \cite{BhattAbsoluteIntegralClosure}.  

In this note we show that such a factorization is a formal consequence of any sufficiently functorial assignment of big Cohen-Macaulay algebras.  Our main result is the following.

\begin{mainthm*}[{\autoref{thm.MainTechnicalTheorem}}]
    Suppose $(R, \fram)$ is a reduced Noetherian universally catenary equidimensional local ring, $I \subseteq R$ is an ideal of positive height and $S = R \oplus It \oplus I^2t^2 \oplus \dots = R[It]$ is the associated Rees algebra.  Set $Y = \Proj S$.  Let $\frn$ denote the homogeneous maximal ideal of $S$ and suppose that $\nu : S_{\frn} \to R$ is the map obtained by localizing the projection onto degree $0$ ($S \to R$) at $\frn$. 
    
    Suppose there is a balanced big Cohen-Macaulay $S_{\frn}$-algebra $B_S$, a  balanced big Cohen-Macaulay $R$-algebra $B_R$, and an $S$-module map  $B_S \to B_R$ such that  
    \[
        \xymatrix{
            S_{\frn} \ar[d] \ar[r]^{\nu} & R \ar[d] \\
            B_S \ar[r] & B_R.
        }
    \]
    commutes.  Then the map $R \to B_R$ factors as 
    \[
        R \to \myR\Gamma(Y, \cO_Y) \to B_R.
    \]
\end{mainthm*}

In particular, given a sufficiently weakly functorial balanced big Cohen-Macaulay algebra assignment (for instance, one that is weakly functorial for surjections from Rees algebras), we have that $R \to B_R$ factors through $\myR\Gamma(Y, \cO_Y)$ for any proper birational map $Y \to \Spec R$.  See \autoref{cor.MainCorollary}.  

There was previous evidence that something like this might be true.  For instance, in \cite[Proposition 3.7]{MaSchwedeSingularitiesMixedCharBCM} it was shown that if $(R, \fram)$ is local of dimension $d$ and one has $B_R$ and $B_S$ as in the diagram above, then the injectivity of $H^d_{\fram}(R) \to H^d_{\fram}(B_R)$ implies the injectivity of $H^d_{\fram}(R) \to H^d_{\fram}(\myR\Gamma(Y, \cO_Y)) = H^d_{\pi^{-1}\fram}(Y, \cO_Y)$.  Hence, we recover this result and also another proof of the fact that weakly BCM-regular rings\footnote{Meaning that $R \to B$ is pure for every BCM $R$-algebra $B$.} are birational derived splinters (for relevant background see \cite{KovacsRat,MaSchwedeSingularitiesMixedCharBCM,Lyu.BirationalDerivedSplinters,MaMcDonaldRGSchwede.BSProperty}).  We emphasize, however, that some aspects of the proof are similar.  When combining this theorem with \cite{MaMcDonaldRGSchwede.BSProperty}, this also recovers the main result of \cite{RodriguezVillalobosSchwede.BrianconSkodaProperty}.  Again though, some techniques from the proof are similar.

\subsection{AI Acknowledgements}
Some key aspects of this paper were suggested by OpenAI's ChatGPT (GPT-5.5 Pro).  In particular, the LLM contributions were large enough that it would have warranted co-authorship had the same suggestions been made by a human.  Specifically, a special case of \autoref{thm.LocalCohomMapIsZero}, the case when $I$ was $\mathfrak{m}$-primary, was found by ChatGPT 5.5 Pro when prompted to look for a proof of the main theorem via a Sancho de Salas argument.  The generalization to the full statement of \autoref{thm.LocalCohomMapIsZero} was identified by the human author, who noted that it might be sufficient to prove the general theorem.

Completing the proof of the desired theorem was then posed to the LLM, which initially proposed several incorrect ways to proceed (notably, it asserted several times that a map in the derived category of a non-affine scheme being zero at stalks implies it is zero).  After some back-and-forth, the LLM suggested a version of a Cousin filtration argument, a variant of which is found below.  Even within that argument, the LLM also made other errors.  For instance, the LLM implicitly asserted \autoref{clm.InjectiveMapOnDR.lem.MainTechnicalLemma} for any $L^{\mydot}$, without any hypothesis on the dimension of the supports of its cohomology.  However, these errors were caught by the human author, and fixed, frequently in a collaborative manner with the LLM.  

We also note that for many of the arguments below, the LLM suggested different approaches than what is found below.  The human author preferred their own arguments.  The final arguments and all writing in this paper were entirely produced by the human author based upon their own understanding and verification of the result.  

OpenAI's ChatGPT (GPT-5.5 Pro) and Anthropic's Claude Fable 5 have also read the text below for accuracy and grammar.  Some of the suggestions were implemented.

\subsection{Acknowledgements}

The author thanks Bhargav Bhatt, Srikanth Iyengar, Linquan Ma, and Sandra Rodr\'iguez-Villalobos for valuable conversations and comments.  We particularly thank Bhargav Bhatt for allowing the inclusion of the alternate proof of \autoref{lem.MainTechnicalLemma} and thank Linquan Ma for pointing out a way to bypass \autoref{lem.MainTechnicalLemma} entirely in some cases.  The author was supported by NSF Grant DMS-2501903 and Simons Foundation Travel Support for Mathematicians SFI-MPS-TSM-00013051.

\section{Preliminaries}

We begin with a brief recap of the theory of big Cohen-Macaulay algebras.

\subsection{Big Cohen-Macaulay modules and algebras}

We recall several variants of the big Cohen-Macaulay condition.

\begin{definition}[Balanced big Cohen-Macaulay modules]
    Suppose $(R, \fram)$ is a Noetherian local ring.  An $R$-module $M$ is called \emph{weakly balanced big Cohen-Macaulay} if every system of parameters of $R$ is a weakly regular sequence\footnote{meaning that $x_{i+1}$ is a non-zerodivisor on $M/(x_1, \dots, x_i)M$} on $M$.
    We say it is \emph{balanced big Cohen-Macaulay} if it is weakly balanced big Cohen-Macaulay and $\fram M \neq M$.  
\end{definition}

We recall that both weakly balanced and balanced big Cohen-Macaulay $R$-modules behave well under localization.

\begin{lemma}[{\cite{Takeuchi.LocalizationsOfBalancedBCM}, \cite[Corollary 2.5]{Zarzuela.BalancedBCMAndFlat}, \cite{Sharp.CMPropertiesForBalancedBigCM}}]
    \label{lem.WeaklyBCMLocalizes}
    Suppose $(R, \fram)$ is a Noetherian universally catenary and equidimensional local ring.  If $M$ is a (weakly) balanced big Cohen-Macaulay $R$-module, then $M_Q$ is a (weakly) balanced big Cohen-Macaulay $R_Q$-module for every $Q \in \Spec R$ such that $M_Q \neq 0$.
\end{lemma}
\begin{proof}
    We sketch the case that $M$ is \emph{weakly} balanced big Cohen-Macaulay as that case is not stated in the literature (although of course, the proof is easier).  For any system of parameters $({x_1}, \dots, {x_r})$ for $R_Q$, we can find $(y_1, \dots, y_r)$, part of a system of parameters for $R$, such that $(y_1, \dots, y_i)R_Q = (x_1, \dots, x_i)R_Q$ for all $i$, see \cite[Lemma 2.2]{Zarzuela.BCMAndLocalization}.

    By hypothesis $y_{i+1}$ is not a zero divisor on $M/(y_1, \dots, y_i) M$ and hence also not a zero divisor on the localization.  By assumption $y_{i+1} R_Q/(y_1, \dots, y_i)R_Q = x_{i+1} R_Q/(x_1, \dots, x_i)R_Q$ and so multiplication by $x_{i+1}$ and $y_{i+1}$ agree up to a unit of $R_Q$ on the localized quotient, and so the result follows.
\end{proof}

We also have the related notion detected by local cohomology.

\begin{definition}
    Suppose $R$ is a Noetherian ring.  An $R$-module $M$ is called \emph{weakly cohomologically Cohen-Macaulay} if for each prime ideal $Q \in \Spec R$ we have that $H^i_Q(M_Q) = 0$ for all $i < \dim R_Q$.   
\end{definition}

Assuming $R$ is local, catenary and equidimensional, it is not difficult to see that weakly balanced big Cohen-Macaulay modules are weakly cohomologically Cohen-Macaulay (as they localize and the argument from the Noetherian case applies).  Bhatt proved the converse.

\begin{lemma}[{\cite[Corollary 2.8]{BhattAbsoluteIntegralClosure}}]
    Suppose $R$ is a Noetherian catenary and equidimensional local ring.  Suppose $M$ is an $R$-module.  If $M$ is weakly cohomologically Cohen-Macaulay then $M$ is weakly balanced big Cohen-Macaulay.  
    
    As a consequence, if $R$ is not necessarily local but is Noetherian, catenary and locally equidimensional, then $M$ is weakly cohomologically Cohen-Macaulay if and only if $M_Q$ is a weakly balanced big Cohen-Macaulay $R_Q$-module for all $Q \in \Spec R$.  
\end{lemma}

\begin{definition}
    We have a \emph{weak CM-assignment} on a class of catenary locally equidimensional Noetherian rings with specified permitted morphisms if the following conditions are satisfied.
    \begin{enumerate}
        \item for every Noetherian ring $R$ in that class, there exists a weakly cohomologically Cohen-Macaulay $R$-algebra $B_R$, and
        \item for every map of rings $S \to R$ in that class (for us it will suffice to consider surjections) there exists a commutative diagram of maps of rings
        \[
            \xymatrix{
                S \ar[d] \ar[r] & R \ar[d] \\
                B_S \ar[r] & B_R.
            }
        \]
    \end{enumerate} 
\end{definition}

\subsection{Cousin complexes}
\label{subsec.CousinComplexes}
    We recall the following work on Cousin complexes and their relations with balanced big Cohen-Macaulay modules.  For more detailed background on Cousin complexes, we refer the reader to \cite[Chapter IV, Section 2]{HartshorneResidues} and \cite{Sharp.CousinForNoetherian, Sharp.CousinComplexBCM,Sharp.LocalCohomologyCousinComplex}.

    Suppose that $R$ is a Noetherian universally catenary locally equidimensional ring of finite dimension $n$ and that $M$ is an $R$-module.   
    We have the following filtration $D_{\bullet}$ of $\Spec R$
    \[
        D_i := \{ \frq \in \Spec R \;|\; \mathrm{ht}\, \frq \geq i \}.
    \]
    That is, $D_0 = \Spec R$, $D_{n+1} = \emptyset$, and if $R$ is local with maximal ideal $\fram$ then $D_n = \{ \fram \}$.
    Furthermore, when $R$ is local our hypotheses imply $R$ is catenary and equidimensional and so we see that $D_i$ can also be written as $\{ \frq \in \Spec R \;|\; \dim R/\frq \leq \dim R - i \}$.  In particular, in the local equidimensional catenary case, the height filtration is also the dimension filtration and so the associated Cousin complex which we construct below can be viewed as based on either filtration.

    To each $i$, we also associate $\partial D_i = D_i \setminus D_{i+1}$.  In particular, $\partial D_0$ is the minimal primes of $R$.

    Associated to this filtration, we have the following \emph{augmented Cousin complex} $C^{\mydot}_M := \mathcal{C}(D_{\bullet}, M)$ for any module $M$.
    \[
        0 \to C^{-1} \xrightarrow{d^{-1}} C^0 \xrightarrow{d^0} C^1 \xrightarrow{d^1} \dots \xrightarrow{d^{n-1}} C^n \to 0
    \]
    where $C^{-1} = M$ and $C^0 := \bigoplus_{\frq \in \partial D_0} M_{\frq}$ and the map $C^{-1} \to C^0$ is simply the obvious localization on each component.
    
    Then in general, for $i \geq 1$, we define 
    \[
        C^i_M = C^i := \bigoplus_{\frq \in \partial D_i} (\coker d^{i-2})_{\frq} =: \bigoplus_{\frq \in \partial D_i} C^i(\frq)
    \]
    where the map $C^{i-1} \to C^i$ is induced by the map $C^{i-1} \to (\coker d^{i-2})_{\frq}$ defined by sending $x \in C^{i-1}$ to  
    \[
        \overline{x}/1 \in (\coker d^{i-2})_{\frq} = \Big({C^{i-1} \over d^{i-2}(C^{i-2})}\Big)_{\frq}
    \]
    for each $\frq \in \partial D_i$.  

    The \emph{non-augmented Cousin complex}, which we denote by $\mathcal{C}'(D_{\bullet}, M)$, is then the stupid (also called brutal) truncation
    \[
        \sigma_{\geq 0} \mathcal{C}(D_{\bullet}, M) = \big(0 \to C^0 \to C^1 \to \dots \to C^n \to 0\big)
    \]
    and so we have a canonical map $M \to \mathcal{C}'(D_{\bullet}, M)$.

    We recall the following facts that we will need later.
    \begin{lemma}
        \label{lem.CousinBasics}
        Using the notation above, and still assuming $R$ is locally equidimensional and universally catenary and still using the height filtration $D_{\bullet}$, we have the following.
        \begin{enumerate}
            \item If $(R, \fram)$ is additionally local and $\dim \Supp M = n$, then $H^n_{\fram}(M)$ is naturally isomorphic to the last term of the Cousin complex, $C^n$.  Furthermore, $\myR\Gamma_{\fram}(C^i) = 0$ for all $i < n$. (\cite[Theorem 1.8]{Sharp.CousinComplexBCM}, \cite[Theorem]{Sharp.LocalCohomologyCousinComplex})  \label{lem.CousinBasics.TopIsLocalCohom}  
            \item If $\frq \in \Spec R$  then $(C^{\mydot}_M)_{\frq} = C^{\mydot}_{M_{\frq}}$. (see \cite[Theorem 3.5]{Sharp.CousinForNoetherian})  \label{lem.CousinBasics.CommutingWithLocalization}  
            \item We have $\Supp (\coker d^{i-2}) \subseteq D_i$, and $C^i_{\frq} = 0$ if $\mathrm{ht}\, \frq < i$, or equivalently if $R$ is additionally local, if $\dim R / \frq > n- i$.  (\cite[Proposition 2.5(iii), (2.7)]{Sharp.CousinForNoetherian})     \label{lem.CousinBasics.SuppOfCousinPiece}
        \end{enumerate}
    \end{lemma}

    Sharp characterized weakly balanced big Cohen-Macaulay modules via Cousin complexes.

    \begin{theorem}[{\cite[Theorem 3.6]{Sharp.CousinComplexBCM}}]
        \label{thm.SharpViaCousin}
        Suppose $(R, \fram)$ is a Noetherian local equidimensional and catenary ring and $M$ is an $R$-module.  
        Then $M$ is a weakly balanced big Cohen-Macaulay module if and only if the augmented Cousin complex  $\mathcal{C}(D_{\bullet}, M)$ is exact.  That is, if and only if 
        \[
            M \to \mathcal{C}'(D_{\bullet}, M)
        \]
        is an isomorphism in the derived category.
    \end{theorem}

    Rephrasing Sharp's result outside the local setting, we have the following.
    \begin{corollary}
        \label{cor.SharpCousinCMForNonLocal}
        Suppose $R$ is Noetherian, universally catenary and locally equidimensional of finite dimension.  Then $M$ is weakly cohomologically Cohen-Macaulay if and only if $M \to \mathcal{C}'(D_{\bullet}, M)$ is an isomorphism in the derived category  where $D_{\bullet}$ denotes the height filtration.
    \end{corollary}

\section{The BCM condition and Rees algebras}

Suppose $(R, \fram)$ is a reduced equidimensional universally catenary Noetherian local ring of dimension $n$, $I$ is an ideal of positive height and $S = R \oplus It \oplus I^2 t^2 \oplus \dots$ is the Rees algebra.  Set $Y = \Proj S$.  Fix $K^{\mydot}$ fitting into the triangle:
\[
    K^{\mydot} \to R \to \myR\Gamma(Y, \cO_Y) \xrightarrow{+1}.
\]
In other words $K^{\mydot}$ is the fiber of $R \to \myR\Gamma(Y, \cO_Y)$.  We claim 
that $\dim \Supp H^j(Y, \cO_Y) \leq n-j-1$ for $j > 0$.  Indeed, if $\dim \Supp H^j(Y, \cO_Y) \geq n-j$, then after localizing at a point in that support of height $j$, we may assume that $\dim R = n = j$.  But $H^n(Y, \cO_Y) = 0$ since the fiber dimension of the birational map $Y \to \Spec R$ is $\leq n-1$.

We then assert that
\begin{equation}
    \label{eq.DimSuppBoundsOnK}
    \dim \Supp \myH^j(K^{\mydot}) \leq n - j.
\end{equation}
To see this, note that if $j \geq 2$, then $\myH^j(K^{\mydot}) = H^{j-1}(Y, \cO_Y)$, and the latter has support of dimension at most $n - (j-1) - 1 = n-j$.  If $j = 1$, then as $Y \to \Spec R$ is birational, $R \hookrightarrow \Gamma(Y, \cO_Y)$ is also birational and hence $\myH^1(K^{\mydot}) = \coker(R \to \Gamma(Y, \cO_Y))$ has support of dimension $\leq n-1$.  Finally, we see that $\myH^0(K^{\mydot}) = \ker (R \to \Gamma(Y, \cO_Y)) = 0$ as $I$ has positive height and $R$ is reduced.    

Fix $\frn = \fram S + S_{> 0}$ to be the homogeneous maximal ideal of $S$.  Let  $\nu : S_{\frn} \to R$ denote the map induced by projection onto the degree zero term.
Suppose that $C$ is an $R$-module and we have a commutative diagram of $S$-modules where $B$ is an $S_{\frn}$-module such that $H^i_{\frn}(B) = 0$ for $i < \dim S_{\frn}$.
\[
    \xymatrix{
        S_{\frn} \ar[d] \ar[r]^{\nu} & R \ar[d] \\
        B \ar[r] & C.
    }
\]
We do not make any Cohen-Macaulay assumptions on $C$ \emph{yet}.

\begin{theorem}
    \label{thm.LocalCohomMapIsZero}
    With notation as above, the composition $\myR\Gamma_{\fram}(K^{\mydot}) \to \myR\Gamma_{\fram}(R) \to \myR\Gamma_{\fram}(C)$ is zero in $D(R)$.
\end{theorem}
\begin{proof}
    The following triangle can be viewed as a variant of the Sancho de Salas sequence
    \[
        [\myR\Gamma_{S_{>0}} S]_0 \to R \to \myR \Gamma(Y, \cO_Y) \xrightarrow{+1},
    \] 
    see for instance \cite[Page 150]{LipmanCohenMacaulaynessInGradedAlgebras}, and so $[\myR\Gamma_{S_{>0}} S]_0 \cong K^{\mydot}$.
    Consider the composition
    $
        \myR\Gamma_{\frn}(S) \cong \myR\Gamma_{\frn}(S_{\frn}) \to S_{\frn} \to B
    $
    which can also be factored as 
    \[
        \myR\Gamma_{\frn}(S) \to \myR\Gamma_{\frn}(B) \to B.
    \]
    As $S_{\frn}$ has dimension $n+1$, we see that $\tau_{\leq n}\myR\Gamma_{\frn}(B) = 0$ and so $\tau_{\leq n} \myR\Gamma_{\frn}(S) \to \myR\Gamma_{\frn}(B)$ is also zero.  Hence we have that 
    \[
        \tau_{\leq n}\myR\Gamma_{\fram}(K^{\mydot}) = \tau_{\leq n}\myR\Gamma_{\fram}([\myR\Gamma_{S_{>0}} S]_0) \to \tau_{\leq n}\myR\Gamma_{\fram}(\myR\Gamma_{S_{>0}}(S)) = \tau_{\leq n}\myR\Gamma_{\frn}(S) \to \myR\Gamma_{\frn}(B)
    \]
    is zero.  But using the map $B \to C$, we see that $\tau_{\leq n}\myR\Gamma_{\fram}(K^{\mydot}) \to \myR\Gamma_{\frn}(C) \to \myR\Gamma_{\fram}(C)$ is also zero.  Thus, it suffices to show that $\myR\Gamma_{\fram}(K^{\mydot})$ only has cohomology in degrees $\leq n$ so that $\tau_{\leq n} \myR\Gamma_{\fram}(K^{\mydot}) \cong \myR\Gamma_{\fram}(K^{\mydot})$.  In view of the triangle defining $K^{\mydot}$, it suffices to observe that $H^n_{\fram}(R) \to H^n_{\fram}(\myR\Gamma(Y, \cO_Y))$ surjects.  But that follows from \autoref{eq.DimSuppBoundsOnK}. Alternatively, assuming the existence of a dualizing complex, that map is  Matlis dual to $\Gamma(Y, \omega_{Y}) \to \omega_{R}$ which injects since $Y \to \Spec R$ is birational and $R$ is reduced.  This completes the proof.
\end{proof}

   The following lemma is the tool which will turn \autoref{thm.LocalCohomMapIsZero} into our main result.

\begin{lemma}
    \label{lem.MainTechnicalLemma}
    Suppose $(R, \fram)$ is a Noetherian equidimensional universally catenary local ring of dimension $n$ and $L^{\mydot} \in D^b_{fg}(R)$ has the property that $\dim \Supp \myH^j(L^{\mydot}) \leq n-j$ for all $j$.  Suppose $M$ is a weakly cohomologically Cohen-Macaulay $R$-module and $f : L^{\mydot} \to M$ is such that 
    \[
        \myR\Gamma_{\frq}(f_{\frq}) : \myR\Gamma_{\frq}(L^{\mydot}_\frq) \to \myR\Gamma_{\frq}(M_{\frq})
    \]
    is the zero map in $D(R)$ for each $\frq \in \Spec R$.  Then $f = 0$ in $D(R)$ as well.
\end{lemma}
The proof we provide first follows a strategy suggested by the LLM.  Below, we also provide a sketch of an alternate proof explained to us by Bhargav Bhatt and used with his permission.
\begin{proof}
    By Sharp's result \cite{Sharp.CousinComplexBCM}, \autoref{thm.SharpViaCousin}, we can replace $M$ by the associated non-augmented Cousin complex $C'^{\mydot}$ induced by the height filtration as in \autoref{subsec.CousinComplexes}.  
    For each integer $i \in \bZ$, consider the triangle induced by the stupid truncations
    \[
        \sigma_{\geq i+1} C'^{\mydot} \to \sigma_{\geq i} C'^{\mydot} \to C'^{i}[-i] \xrightarrow{+1},
    \]
    see for instance \cite[\href{https://stacks.math.columbia.edu/tag/0118}{Tag 0118}]{stacks-project}.  
    At the generic points $\frq$ of $\Spec R$ (corresponding to irreducible components), the hypothesis guarantees that $L^{\mydot}_{\frq} \cong \myR\Gamma_{\frq}(L^{\mydot}_\frq) \to \myR\Gamma_{\frq}(M_{\frq}) \cong C'^0_{\frq}$ is the zero map.  As there are finitely many such minimal primes $\frq$, and $C'^0 = \oplus_{\frq} M_{\frq}$ by construction, we see that $L^{\mydot} \to M \to C'^0$ is zero.  
    
    We will show by induction that we have a factorization of $f$, $L^{\mydot} \to \sigma_{\geq i} C'^{\mydot} \to C'^{\mydot}$.  The argument above is the first step.  Indeed, we just showed that the solid diagonal map in the diagram below is zero and hence the dotted arrow below exists:
    \[
    \begin{tikzcd}
& L^{\mydot} \ar[rd, "\mu"] \ar[d] \ar[dl, dotted]  \\
 \sigma_{\geq 1} C'^{\mydot} \ar[r] & C'^{\mydot} \ar[r] & C'^0 \ar[r, "+1"] & {}.
\end{tikzcd}
\]
In particular, we now have a map $L^{\mydot} \to \sigma_{\geq 1} C'^{\mydot}$.  Proceeding by induction, suppose we have a factorization $f : L^{\mydot} \xrightarrow{f_i} \sigma_{\geq i} C'^{\mydot} \to C'^{\mydot}$.  We will show that $f$ also factors through $\sigma_{\geq i+1} C'^{\mydot}$.  This will complete the proof as $\sigma_{\geq n+1} C'^{\mydot} = 0$.

By induction, we have the following diagram
\begin{equation}
    \label{eq.InductiveDiagonalMap}
    \begin{tikzcd}
        & L^{\mydot} \ar[d,"f_i"'] \ar[dr] & \\
        \sigma_{\geq i+1} C'^{\mydot} \ar[r] & \sigma_{\geq i} C'^{\mydot} \ar[r] & C'^i[-i] \ar[r, "+1"] & {}.
    \end{tikzcd}
\end{equation}

We now prove two claims.  

\begin{claim}
    \label{clm.InjectiveMapOnDR.lem.MainTechnicalLemma}
    The map $\Hom_{D(R)}(L^{\mydot}_{\frq}, C_{\frq}^i[-i]) \to \Hom_{D(R)}(\myR\Gamma_{\frq}(L^{\mydot}_{\frq}), C_{\frq}^i[-i])$ is injective for each $\frq \in \partial D_i$.
\end{claim}
\begin{proof}[Proof of claim]
    Fix $\frq \in \partial D_i$ and set $U$ to be  $\Spec R_{\frq} \setminus \{\frq R_{\frq}\}$.  In view of the triangle 
    $\myR \Gamma_{\frq}(L_{\frq}^{\mydot}) \to L_{\frq}^{\mydot} \to \myR\Gamma(U, \widetilde{L}^{\mydot}_{\frq}) \xrightarrow{+1}$
    it suffices to prove that $\Hom_{D(R)}(\myR\Gamma(U, \widetilde{L}^{\mydot}_{\frq}), C_{\frq}^i[-i]) = 0$ (here $\widetilde{L}^{\mydot}$ denotes the corresponding complex of quasi-coherent $\cO_{\Spec R}$-modules on the affine scheme $\Spec R$).  For that, it suffices to show that $\myR\Gamma(U, \widetilde{L}_{\frq}^{\mydot}) \in D^{< i}$.
    
    Notice that $\dim U  = \dim R_{\frq} - 1 = n - (n - i) - 1 < i$ as we removed the unique closed point of $\Spec R_{\frq}$.  
    Furthermore, by hypothesis and the catenary and equidimensional hypotheses, 
    we know that $\dim \Supp \myH^j(L^{\mydot}_{\frq}) \leq i - j$ and hence similarly $\dim \Supp \myH^j(\widetilde{L}^{\mydot}_{\frq}|_U) \leq i - j - 1$ viewed as a subset of $U$ (closed subsets of $\Spec R_{\frq}$ lose one dimension when removing the closed point).  Hence $H^k(U, \myH^j (\widetilde{L}^{\mydot}_{\frq})) = 0$ if $k \geq i - j$, or in other words if $k + j \geq i$.  But then a spectral sequence argument implies that
    \[
        \myH^{\geq i}(\myR\Gamma(U, \widetilde{L}_{\frq}^{\mydot})) = 0
    \]
    as desired, proving the claim.
\end{proof}

\begin{claim}
    \label{clm.MainInductiveStep.lem.MainTechnicalLemma}
    The map $L^{\mydot} \to C'^i[-i]$ is zero.
\end{claim}
\begin{proof}[Proof of claim]
    As $C'^{i}[-i]$ is a shifted direct sum of $R_{\frq}$-modules $C^i(\frq)$ for $\frq \in \partial D_i$, $L^{\mydot} \in D^b_{fg}(R)$, and $R$ is Noetherian, we see that it suffices to prove that each induced $L^{\mydot} \to C^i(\frq)[-i]$ is zero.  This map factors through $\mu_{\frq} : L_{\frq}^{\mydot} \to C^i(\frq)[-i]$ as the target is already an $R_{\frq}$-module.  It suffices to show that $\mu_{\frq}$ is zero.
    
    If $\dim \Supp M_{\frq} < i$ then $C^i(\frq) = 0$ already by construction and there is nothing to show, so we may suppose that $\dim \Supp M_{\frq} = i$.
    Now, as $C^i(\frq)$ is supported at the maximal ideal of $R_{\frq}$, we see that $\myR\Gamma_{\frq}(C^i(\frq)[-i]) \to C^i(\frq)[-i]$ is an isomorphism.  Consider the canonical map $C^i(\frq)[-i] \to C'^{\mydot}_{\frq} \cong M_{\frq}$ in the derived category and apply  $\myR\Gamma_{\frq}(-)$.  Now, $M$ is weakly cohomologically Cohen-Macaulay and $\dim \Supp M_{\frq} = i$ so that by \autoref{lem.CousinBasics} \autoref{lem.CousinBasics.TopIsLocalCohom} we see that this map is an isomorphism:
    \[
         \myR\Gamma_{\frq}(C^i(\frq)[-i]) \cong \myH^i_{\frq}(M_{\frq})[-i] \cong \myR\Gamma_{\frq}(M_{\frq}).
    \]
    Putting this together we can view $\myR\Gamma_{\frq}(\mu_{\frq})$ as 
    \[
        \myR\Gamma_{\frq}(L^{\mydot}_{\frq}) \to \myR\Gamma_{\frq}(M_{\frq}) \cong \myR\Gamma_{\frq}(C^i(\frq)[-i]) \cong C^i({\frq})[-i].
    \]
    That map is zero by hypothesis and then $\mu_{\frq} : L^{\mydot}_{\frq} \to C^i({\frq})[-i]$ is zero by \autoref{clm.InjectiveMapOnDR.lem.MainTechnicalLemma}.
\end{proof}

    By \autoref{clm.MainInductiveStep.lem.MainTechnicalLemma}, we see that $L^{\mydot} \xrightarrow{f_i} \sigma_{\geq i} C'^{\mydot} \to C'^{i}[-i]$ is zero and hence $f_i$ factors through $\sigma_{\geq i+1} C'^{\mydot} \to \sigma_{\geq i} C'^{\mydot}$.  This completes the proof. 
\end{proof}

Now we sketch the alternate proof provided by Bhatt which uses the middle perverse $t$-structure as developed in \cite{Gabber.tStruc}, see also \cite[Section 3]{BMPSTWW2} for a friendly summary in a similar case.  Thus we also assume that $R$ has a dualizing complex (or that we can reduce to that case) as that is required in the above sources.  We recall the following definitions for $M^{\mydot} \in D^b(R)$.
\begin{itemize}
    \item{} $M^{\mydot} \in {}^p D^{\leq 0}(R)$ if for each $\frq \in \Spec R$, we have $M^{\mydot}_{\frq} \in D^{\leq -\dim R/\frq}$.  This can also be checked by verifying that $\myR\Gamma_{\frq}(M^{\mydot}_{\frq}) \in D^{\leq -\dim R/\frq}$ for every $\frq$ (see the proof of \cite[Lemma 3.2]{BMPSTWW2}).
    \item{} $M^{\mydot} \in {}^p D^{\geq 0}(R)$ if for each $\frq \in \Spec R$, we have $\myR\Gamma_{\frq}(M^{\mydot}_{\frq}) \in D^{\geq -\dim R/\frq}$.
\end{itemize}
We note that we have a functor $\Psi_{\frq}(-) = \myH^{-\dim R/\frq}(\myR\Gamma_{\frq}(-_{\frq}))$.  Observe that $\Psi_{\frq}(-)$ is an exact functor from the perverse heart $D^b(R)^{\heartsuit}$ to $\mathrm{Mod}_{R_{\frq}}$.  We also note that $\Psi_{\frq}(-)$ can be viewed as $\myR\Gamma_{\frq}(-_{\frq})[-\dim R/\frq]$ when acting on  $D^b(R)^{\heartsuit}$ since $\myR\Gamma_{\frq}(-_{\frq})$ of a perverse object lives in a single cohomological degree (see \cite{Gabber.tStruc} and \cite[Section 3]{BMPSTWW2}).

\begin{proof}[{Alternate proof of \autoref{lem.MainTechnicalLemma} \lbrack Bhatt\rbrack}]
    We have that $L' := L^{\mydot}[n] \in {}^p D_{fg}^{\leq 0}(R)$ (that is, $L^{\mydot}[n]$ is coherent perverse connective) and that $M' := M[n] \in {}^p D^{\leq 0}(R) \cap {}^p D^{\geq 0}(R)$ (that is, $M[n]$ is perverse).  
    Consider the factorization $f' : L' \to {}^p \myH^0(L') \to {}^p \myH^0 (M') = M'$.  We claim we can replace $L'$ by ${}^p \myH^0(L')$.  The latter is certainly perverse so to prove the claim it suffices to prove that
    \[
        \Psi_{\frq}(L') \to \Psi_{\frq}({}^p \myH^0(L'))
    \]
    surjects.  Consider the perverse truncation triangle
    \[
        {}^p \tau_{\leq -1} L' \to L' \to {}^p{\myH}^0(L') \xrightarrow{+1}
    \]
    We must show that $H^{-\dim R/\frq+1}_{\frq}(({}^p \tau_{\leq -1} L')_{\frq}) = 0$.  But this follows from the fact that ${}^p \tau_{\leq -1} L'\in {}^p D^{\leq -1}$, and we also have that $H^{-\dim R/\frq}_{\frq}(({}^p \tau_{\leq -1} L')_{\frq}) = 0$.  Hence the desired map is not just surjective, it is an isomorphism.
    Thus we may replace $L'$ with its perverse $0$th cohomology and so assume that $L'$ is also perverse.
    
    Therefore, we can work in the heart of this $t$-structure, an Abelian category.  We want to prove that the map $f' : L' \to M'$ is zero.   
    Let $B'$ be the perverse image of $L' \to M'$, it suffices to show that $B'$ is zero.  Taking the perverse kernel $K'$ and cokernel $C'$ of $f'$ and applying $\Psi_{\frq}(-)$, we have short exact sequences 
        \[ 
            0 \to \Psi_{\frq}(K') \to \Psi_{\frq}(L') \to \Psi_{\frq}(B') \to 0
            \;\;\;\;\text{ and } \;\;\;\;        
            0 \to \Psi_{\frq}(B') \to \Psi_{\frq}(M') \to \Psi_{\frq}(C') \to 0.
        \]  
        Thus since $\Psi_{\frq}(L') \twoheadrightarrow \Psi_{\frq}(B') \hookrightarrow \Psi_{\frq}(M')$ is zero, we see that $\Psi_{\frq}(B')$ is zero for all $\frq$.  As $B'$ is perverse, we see that $\myR\Gamma_{\frq}(B'_{\frq}) = 0$ for all $\frq \in \Spec R$.
    
    It thus suffices to show that an arbitrary object $A^{\mydot} \in D(R)$ is zero if $\myR\Gamma_{\frq}(A^{\mydot}_{\frq}) = 0$ for all $\frq \in \Spec R$.  But that can be checked via induction on $\dim R_{\frq}$ and the triangle 
    \[
         \myR\Gamma_{\frq}(A^{\mydot}_{\frq}) \to A^{\mydot}_{\frq} \to \myR\Gamma(U, \widetilde{A^{\mydot}_{\frq}}) \xrightarrow{+1}
    \]
    where $U = \Spec R_{\frq} \setminus \{ \frq R_{\frq} \}$.  By induction, $A^{\mydot}_{\frp} = 0$ for all $\frp \in U$, so that  $\widetilde{A^{\mydot}_{\frq}}|_U = 0$ and hence  $\myR\Gamma(U, \widetilde{A^{\mydot}_{\frq}}) = 0$.  But then the hypothesis about vanishing local cohomology implies that $A^{\mydot}_{\frq} = 0$ and so the proof is complete by induction.
\end{proof}

\begin{theorem}
    \label{thm.MainTechnicalTheorem}
    With notation as in the start of the section, suppose that there is a commutative diagram of $S$-modules
    \[
        \xymatrix{
            S_{\frn} \ar[d] \ar[r] & R \ar[d] \\
            B_S \ar[r] & B_R
        }
    \] 
    where $B_S$ and $B_R$ are weakly balanced big Cohen-Macaulay $S_{\frn}$- and $R$-algebras, respectively.  
    Then we have a factorization 
    \[
        R \to \myR\Gamma(Y, \cO_Y) \to B_R.
    \]
\end{theorem}
\begin{proof}
    Setting $K^{\mydot}$ as at the start of the section, it suffices to show that $K^{\mydot} \to B_R$ is zero.  Notice that if $\frq \in \Spec R$, then $S_{\frq} = S \otimes_R R_{\frq}$ is the Rees algebra of $R_{\frq}$ along $I_{\frq}$.  If $\frn_{\frq}$ is the prime ideal of $S$ corresponding to the homogeneous maximal ideal of $S_{\frq}$, then we see that we have a diagram 
    \[
        \begin{tikzcd}
            S_{\frn_{\frq}} \ar[r] \ar[d] & R_{\frq} \ar[d] \\
            (B_S)_{\frn_{\frq}} \ar[r] & (B_R)_{\frq}
        \end{tikzcd}
    \]
    where the bottom row is either a map of weakly balanced big Cohen-Macaulay algebras or $(B_R)_{\frq} = 0$, see \autoref{lem.WeaklyBCMLocalizes}.
    Thus \autoref{thm.LocalCohomMapIsZero} implies that $\myR\Gamma_{\frq}(K^{\mydot}_\frq) \to \myR\Gamma_{\frq}( (B_R)_{\frq})$ is zero for all $\frq \in \Spec R$ (note if $(B_R)_{\frq}$ is zero, there is nothing to show).  But then, by \autoref{lem.MainTechnicalLemma} and \autoref{eq.DimSuppBoundsOnK}, the result follows.
\end{proof}

\begin{remark}
    Linquan Ma pointed out that if one additionally assumes that the balanced big Cohen-Macaulay module $B_R$ is $\fram$-adically complete (and so derived $\fram$-adically complete), then one can avoid \autoref{lem.MainTechnicalLemma} completely.  Indeed \autoref{thm.LocalCohomMapIsZero} combined with \cite[\href{https://stacks.math.columbia.edu/tag/0A6Y}{Tag 0A6Y}]{stacks-project} immediately implies that the derived completion (denoted ${-}^{\wedge})$
    \[
        K^{\mydot \wedge} \to  R^{\wedge} \to B_R^{\wedge}
    \]
    is zero.  Since $B_R = B_R^{\wedge}$ the result follows.
\end{remark}

\begin{corollary}
    \label{cor.MainCorollary}
    Suppose $R$ is a reduced universally catenary locally equidimensional Noetherian ring.  Suppose further we have a weak CM-assignment for reduced universally catenary locally equidimensional rings $R \mapsto B_R$ (or even simply for surjections from essentially of finite type reduced locally equidimensional $R$-algebras).  Then for every proper birational map $Y \to \Spec R$, there is a factorization in $D(R)$,
    \[
        R \to \myR\Gamma(Y, \cO_Y) \to B_R.
    \]    
\end{corollary}

\begin{proof}
    Assuming $Y \to \Spec R$ is proper birational, by Chow's lemma, we can dominate $Y \to \Spec R$ by a projective birational map $Y' \to Y \to \Spec R$ so that we have a factorization $R \to \myR\Gamma(Y, \cO_Y) \to \myR\Gamma(Y', \cO_{Y'})$.  It thus suffices to prove the result for $Y'$ and hence replacing $Y$ by $Y'$, we may assume that $Y$ is the blowup of some ideal of positive height.  Letting $K^{\mydot} \to R \to \myR\Gamma(Y, \cO_Y) \xrightarrow{+1}$ be an exact triangle, it suffices to show that the composition $K^{\mydot} \to R \to B_R$ is zero.  But this can be checked after localization.  The result then follows from \autoref{thm.MainTechnicalTheorem}.
%
\end{proof}


\begin{remark}
    Frequently $B_R$ is chosen to be an $R^+$-algebra.  Under moderate hypotheses, it is possible to obtain the same factorization for $Y \to \Spec R$ an alteration by using a Stein factorization.  We do not do this here.
\end{remark}

Based on the behavior of $R^+$ and $\widehat{R^+}$ in positive and mixed characteristic (\cite{BhattDerivedDirectSummand,BhattAbsoluteIntegralClosure}), one might expect the following stronger statement that we do not address.

\begin{question}
    Assuming $B_R$ is obtained from a weakly functorial balanced big Cohen-Macaulay algebra assignment, is there a factorization $R \to \myR\Gamma(Y, \cO_Y) \to B_R$ where the maps are of \emph{derived commutative algebras}?  
\end{question}

\bibliographystyle{skalpha}
\bibliography{MainBib}

\end{document}